\documentclass[11pt,reqno]{amsart}
\usepackage{fullpage,enumerate}
\usepackage{amsfonts}
\usepackage{amssymb}
\usepackage{times}
\usepackage{graphicx}
\usepackage{amsmath,bm}
\usepackage{tikz-cd}
\usepackage{tikz}
\usepackage{appendix}
\usepackage{color}\usepackage{float}

\usepackage{mathrsfs}
\usepackage{geometry}
\newtheorem{thm}{Theorem}[section]
\newtheorem{cor}[thm]{Corollary}
\newtheorem{lem}[thm]{Lemma}
\newtheorem{prop}[thm]{Proposition}

\theoremstyle{definition}
\newtheorem{ex}[thm]{Example}
\newtheorem{nota}[thm]{Notation}
\newtheorem{defn}[thm]{Definition}

\theoremstyle{remark}
\newtheorem{rem}[thm]{Remark}
\begin{document}

\address{Einstein Institute of Mathematics, Edmond J. Safra Campus, The Hebrew University of Jerusalem,
Givat Ram. Jerusalem, 9190401, Israel}
\email{xiang.he@mail.huji.ac.il}

\title[Short version of title]{Brill-Noether Generality of Binary Curves}
\author{Xiang He}
\maketitle

\begin{abstract}
We show that the space $G^r_{\underline d}(X)$ of linear series of certain multi-degree $\underline d=(d_1,d_2)$ (including the balanced ones) and rank $r$ on a general genus-$g$ binary curve $X$ has dimension $\rho_{g,r,d}=g-(r+1)(g-d+r)$ if nonempty, where $d=d_1+d_2$. This generalizes Theorem 24 of Caporaso's paper \cite{caporaso2010brill} from the case $r\leq 2$ to arbitrary rank, and shows that the space of Osserman-limit linear series on a general binary curve has the expected dimension, which was known for $r\leq 2$ (\cite[\S 7]{osserman2014dimension}). In addition, we show that the space $G^r_{\underline d}(X)$ is still of expected dimension after imposing certain ramification conditions with respect to a sequence of increasing effective divisors supported on two general points $P_i\in Z_i$, where $i=1,2$ and $Z_1,Z_2$ are the two components of $X$.     Our result also has potential application to the lifting problem of divisors on graphs to divisors on algebraic curves.
\end{abstract}

\section{Introduction}
Let $k$ be an algebraically closed field. For a curve $C$ over $k$, a $\mathfrak g^r_d$ on $C$ denotes a linear series on $C$ of degree $d$ and rank $r$.  
Brill-Noether theory states that if $C$ is general, then the space $G^r_d(C)$ of $\mathfrak g^r_d$s on $C$ has dimension
$$\rho_{g,r,d}=g-(r+1)(g-d+r).$$
This provides a bridge between the realms of abstract curves and curves in projective spaces: for a general curve, the theory tells us whether it is equipped with a nondegenerate map of a certain degree to a certain projective space. The result was first proved in \cite{griffiths1980variety} by degenerating to an irreducible curve $C'$ with $g$ nodes (the normalization of which is a rational curve), as in the left part of Figure \ref{fig: introduction}, and showing that $G^r_d(C')$ has dimension at most $\rho_{g,r,d}$. After that, it has inspired a great deal of additional research, including re-proofs and new techniques such as limit linear series and tropical linear series  \cite{eisenbud1986limit,cools2012tropical,osserman2019limit,amini2015linear}.


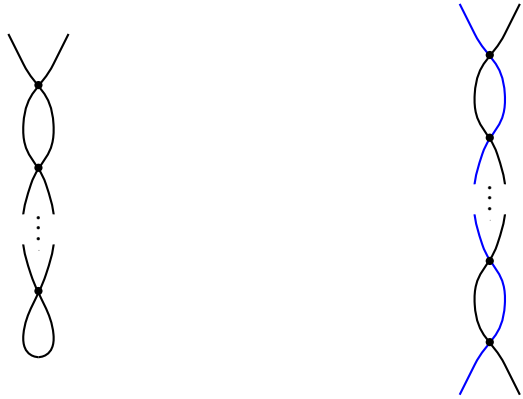
\begin{figure}[H]
\begin{tikzpicture}[scale=0.4]
    
    \draw[thick] plot [smooth,tension=1] coordinates{(1,5)(0.5,4)(-0.5,2)(0.25,0)(0.5,-1)};
    \draw[blue,thick] plot [smooth,tension=1] coordinates{(-1,5)(-0.5,4)(0.5,2)(-0.25,0)(-0.5,-1)};
    \draw[thick] plot [smooth,tension=1] coordinates{(1,-8)(0.5,-7)(-0.5,-5)(0.25,-3)(0.5,-2)};
    \draw[blue, thick] plot [smooth,tension=1] coordinates{(-1,-8)(-0.5,-7)(0.5,-5)(-0.25,-3)(-0.5,-2)};
    \draw (0,-2.2) node[circle, fill=black, scale=0, label=above:{\vdots}]{};
     \draw (0,-3.55) node[circle, fill=black, scale=0.3]{};
     \draw (0,-6.25) node[circle, fill=black, scale=0.3]{};
     
     \draw (0,3.3) node[circle, fill=black, scale=0.3]{};
     \draw (0,0.55) node[circle, fill=black, scale=0.3]{};

     \draw[shift={(-15,-1)}][thick] plot [smooth,tension=1] coordinates{(1,5)(0.5,4)(-0.5,2)(0.25,0)(0.5,-1)};
     \draw[shift={(-15,-1)}][thick] plot [smooth,tension=1]coordinates{(-1,5)(-0.5,4)(0.5,2)(-0.25,0)(-0.5,-1)};
     \draw[shift={(-15,-1)}][thick] plot [smooth,tension=1] coordinates{(0,-5.75)(-0.5,-5)(0.25,-3)(0.5,-2)};
      \draw[shift={(-15,-1)}][thick] plot [smooth,tension=1] coordinates{(0,-5.75)(0.5,-5)(-0.25,-3)(-0.5,-2)};
    \draw [shift={(-15,-1)}](0,-2.2) node[circle, fill=black, scale=0, label=above:{\vdots}]{};
    
     \draw[shift={(-15,-1)}] (0,3.3) node[circle, fill=black, scale=0.3]{};
     \draw[shift={(-15,-1)}] (0,0.55) node[circle, fill=black, scale=0.3]{};
     \draw[shift={(-15,-1)}] (0,-3.55) node[circle, fill=black, scale=0.3]{};
    
\end{tikzpicture}
\caption{A genus $g$ irreducible curve (resp. binary curve) with $g$ (resp. $g+1$) nodes.}\label{fig: introduction}
\end{figure}

While the space $G^r_d(C)$ is well-studied for smooth curves (cf. \cite[Ch.V]{arbarello2013geometry}), little is known for reducible ones. 
In this paper we consider a \textbf{binary curve} $X$ of genus $g$, which, as in the right part of Figure \ref{fig: introduction}, is obtained by glueing two copies $Z_1, Z_2$ of $\mathbb P^1_k$ along $g+1$ pairs of different points $(Q^1_j,Q^2_j)$, where $Q^i_j\in Z_i$ for $0\leq j\leq g$ and $1\leq i\leq 2$. For a multi-degree $\underline d=(d_1,d_2)$ on $(Z_1,Z_2)$ such that $d_1+d_2=d$, denote by $\mathfrak g^r_{\underline d}$ a linear series on $X$ whose underlying line bundle has multi-degree $\underline d$. It was proved in \cite{caporaso2010brill} that the space $G^r_{\underline d}(X)$ of $\mathfrak g^r_{\underline d}$s on a general binary curve $X$ has dimension at most $\rho_{g,r,d}$ if $\underline d$ is balanced, namely $|d_1-d_2|\leq g+1$, and $r\leq 2$. We extend the result of Caporaso to arbitrary rank $r$, and show that $G^r_{\underline d}(X)$ still has expected dimension after imposing a certain ramification conditions:

\begin{thm}\label{thm: introduction}
Let $X$ be a genus-$g$ binary curve obtained from glueing two copies $Z_1,Z_2$ of $\mathbb P^1_k$, and fix $r$ and $\underline d$ as above. Let $P_i\in Z_i$ be two general points. Suppose $d_i\geq -1$ for $1\leq i\leq 2$ and $X$ is general. Then for a sequence $D_\bullet$ of increasing effective divisors supported on $\{P_1,P_2\}$ and an admissible sequence $a_\bullet $ along $D_\bullet$ satisfying certain conditions, the space of $\mathfrak g^r_{\underline d}$s with multi-vanishing sequence at least $a_\bullet$ along $D_\bullet$ is either empty or of expected dimension. 
In particular, $G^r_{\underline d}(X)$ is either empty or of dimension $\rho_{g,r,d}$.
\end{thm}
See Theorem \ref{thm: main theorem} for details. See also Section \ref{subsec:Multi-vanishing} for the notions of admissible sequences and the corresponding ramification conditions.
Our proof of Theorem \ref{thm: introduction} benefits from the idea of \cite{griffiths1980variety}. In their proof, the normalization $\widetilde C'$ of $C'$ is realized as a rational normal curve in $\mathbb P^d_k$. Let $(N^1_j, N^2_j)_{1\leq j\leq g}$ be the preimage in $\widetilde C'$ of each node of $C'$, as in Figure \ref{fig: introduction degeneration}. The space $G^r_d(C')$ is identified with an open subset of the space of $(d-r-1)$ planes of $\mathbb P^d_k$ that pass though the $g$ chords $\overline{N^1_jN^2_j}$ of $\widetilde C'$, where  $\overline{N^1_jN^2_j}$ meets $\widetilde C'$ at $N^1_j$ and $N^2_j$.
Then it is enough to show that the intersection $$\sigma(\overline{N^1_1N^2_1})\cap\cdots\cap\sigma(\overline{N^1_gN^2_g})$$ is dimensionally transverse. Here $\sigma(\overline{N^1_jN^2_j})\subset Gr(d-r,d+1)$ is the Schubert cycle of $(d-r-1)$-planes in $\mathbb P^d_k$ meeting $\overline{N^1_jN^2_j}$. Proceeded by induction, at each step they degenerate $N^1_m$ and $N^2_m$ to $N^1_1$ and get a union of Schubert varieties $S_m$ of the flag $$N^1_1\subset \overline{2N^1_1}\subset\cdots\subset \overline{(d+1)N^1_1}=\mathbb P^d_k,$$ where $\overline {jN^1_1}$ denotes the $j$-th osculating plane of $\widetilde C'$ at $N^1_1$. Then the problem is reduced to the properness of the intersection $S_m\cap\sigma(\overline {N^1_{m+1}N^2_{m+1}})\cap\cdots\cap\sigma(\overline {N^1_gN^2_g}).$

\begin{figure}[H]
\begin{tikzpicture}[scale=0.11]

    \draw[shift={(-70,-3)}][-] (47,-25) -- (55,-48);
    \draw[shift={(-70,-3)}][-] (43,-24) -- (42,-48);
    \draw[shift={(-70,-3)}][-] (40,-20) -- (27,-47);

    \draw[shift={(-70,-3)}][thick] plot [smooth,tension=1] coordinates{(40,-10)(38.5,-24.5)(56,-35)(45,-43)(23,-42)};
    
     \draw [shift={(-70,-3)}](38.2,-23.8) node[circle, fill=black, scale=0.3, label=left:{$N^1_g$}]{};
     
     \draw [shift={(-70,-3)}](42.8,-28.3) node[circle, fill=black, scale=0.3, label=left:{}]{};
     \draw [shift={(-70,-3)}](42.8,-30) node[circle, fill=black, scale=0, label=left:{$N^1_2$}]{};
     
     \draw [shift={(-70,-3)}](48.8,-30.5) node[circle, fill=black, scale=0.3, label=above:{}]{};
      \draw [shift={(-70,-3)}](50,-32.5) node[circle, fill=black, scale=0, label=left:{$N^1_1$}]{};
     
     \draw [shift={(-70,-3)}](52.5,-41) node[circle, fill=black, scale=0.3, label=left:{}]{};
     \draw [shift={(-70,-3)}](54,-44) node[circle, fill=black, scale=0, label=left:{$N^2_1$}]{};
     
     \draw [shift={(-70,-3)}](42.2,-43.5) node[circle, fill=black, scale=0.3, label=above:{}]{};
     \draw [shift={(-70,-3)}](42.2,-46) node[circle, fill=black, scale=0, label=left:{$N^2_2$}]{};
     
     \draw [shift={(-70,-3)}](29,-43) node[circle, fill=black, scale=0.3, label=above:{}]{};
      \draw [shift={(-70,-3)}](27,-43) node[circle, fill=black, scale=0, label=above:{$N^2_g$}]{};
     
     \draw [shift={(-70,-3)}](40,-10) node[circle, fill=black, scale=0, label=above:{$\widetilde C'$}]{};

 \draw [shift={(-70,-3)}](37,-40) node[circle, fill=black, scale=0, label=above:{$\cdots$}]{};

\end{tikzpicture}
\caption{}\label{fig: introduction degeneration}
\end{figure}

In principle, the idea of degenerating the chords of the normalization of $C'$ should also work in our context. However, we inevitably came across Schubert cycles with excess dimension during the degeneration. See Remark \ref{rem:issue} for details. To resolve this issue, we adopt Esteves' compactification \cite{esteves2001compactifying} of the Jacobian of $X$ (see \S \ref{subsec:compactified}), 
and construct a compactification of $G^r_{\underline d}(X)$ accordingly. We observe that the Schubert condition of $S_m$ induces a ramification condition at $N^1_1$ for linear series on $\widetilde C'$, and thus modify the proof of \cite{griffiths1980variety} to our case.  See Section \ref{sec: main} for details. 

As an application, our result shows that the space of Osserman limit linear series on a general binary curve has the expected dimension, which was known for $r\leq 2$ (\cite[Corollary 7.4]{osserman2014dimension}). 

\begin{cor}\label{cor:dimension of lls on binary curves}    
The space of Osserman-limit linear series on a general binary curve has the expected dimension.  As a result all limit linear series are smoothable to nearby  smooth curves.
\end{cor}
\begin{proof}
The first statement follows directly from Theorem \ref{thm: introduction} and \cite[Corollary 7.3]{osserman2014dimension}, while the second one is a result of the main smoothing theorem of limit linear series carried out in \cite{osserman2019limit}.
\end{proof}

Moreover, we comment that the smoothing property we get for limit linear series also provides potential tools for lifting rank-$r$ divisors on the graph $G$ associated to a regular family $\mathcal X$ of curves over a discrete valuation ring to divisors on the generic fiber of $\mathcal X$ that admit the same rank. More precisely, this can be achieved by lifting the divisor on $G$ to a limit linear series on the special fiber, which in our case is a general binary curve, and applying Corollary \ref{cor:dimension of lls on binary curves}. A similar approach can be found in \cite[\S 5]{he2019smoothing}. See also \cite[\S 10]{baker2016degeneration} for a survey on this problem, and \cite{he2018lifting} for results of lifting divisors while preserving both the rank and ramifications. Note that the liftability is only completely solved when $G$ is a chain of loops \cite{cartwright2014lifting}.

\subsection*{Conventions.} All curves we consider, unless otherwise specified, are assumed proper over an algebraically closed field $k$, reduced and connected, and at worst nodal.

\subsection*{Acknowledgements.} I would like to thank Brian Osserman for introducing this problem to me and for helpful conversations. I would also like to thank Eric Larson for making me aware of the idea in \cite{griffiths1980variety} and its connection to our case.

\section{The compactified relative Jacobian and moduli scheme of linear series}
In this section we recall the notion of compactified relative Jacobian over a family of (possibly reducible) curves carried out in \cite{esteves2001compactifying}. For each curve $X$ in the family, the compactified Jacobian contains extra points corresponding to certain torsion free, rank one sheaves on $X$. Note that for a pair $(I,V)$ consisting of a torsion free rank one sheaf $I$ on $X$ and a subspace $V\subset H^0(X,I)$, we can still talk about ramifications of $V$ at nonsingular points of $X$. Then for the degeneration argument of the next section, we construct a family of (compactified) moduli spaces of linear series with imposed ramifications on binary curves as a closed subscheme of a relative Grassmannian over the compactified relative Jacobian. Note that this construction is well-known, see for example \cite[\S 4]{osserman2006limit} or \cite[\S 2]{ossermanlimit}.

\subsection{The compactified relative Jacobian}\label{subsec:compactified}

 Let $X$ be a curve and $ I$ a coherent sheaf on $X$. We say that $I$ is \textbf{torsion free} if it has no embedded components and $I$ has \textbf{rank one} if it has generic rank one at every irreducible component of $X$. We say that $I$ is \textbf{simple} if $\mathrm{End}_X(I)=k$. According to \cite[Proposition 10.1]{oda1979compactifications}, each torsion free sheaf $ I$ of rank one can be written as the pushforward of a line bundle $ L_I$ on a partial normalization $n_I\colon  X_I\rightarrow X$. A direct calculation shows that the natural map
$$n_{I*}\mathscr E\mathrm{nd}(L_I)\rightarrow
\mathscr E\mathrm{nd}(n_{I*}L_I)$$ is an isomorphism. Hence a torsion free sheaf is simple if and only if the corresponding normalization $X_I$ is connected. 

Let $I$ be a torsion free sheaf of rank one on $X$. If $Y\subset X$ is a subcurve, then we denote by $I_Y$ the maximum torsion free quotient of $I|_Y$. Fix an integer $d$, and fix a line bundle $\mathscr E$ on $X$ of degree $d-g+1$. We say that $\mathscr E$ is a \textbf{polarization} on $X$. For every subcurve $Y\subset X$, let $e_Y:=\deg_Y(\mathscr E)$. We say that $I$ is \textbf{stable} (resp. \textbf{semistable}) with respect to $\mathscr E$ if for every nonempty proper subcurve $Y\subsetneq X$, we have 
$$\chi(I_Y)>e_Y\quad (\mathrm{resp.}\ \chi(I_Y)\geq e_Y).$$
Let $P\in X$ be a non-singular point. We say that a semistable sheaf $I$ on $X$ is $P$-\textbf{quasistable} with respect to $\mathscr E$ if $\chi(I_Y)>e_Y$ for all proper subcurves $Y$ of $X$ containing $P$. We refer to \cite{esteves2001compactifying} for more general definitions of polarization and stability. 

\begin{ex}\label{ex: polarization stableness}
Let $X$ be a binary curve with components $Z_1$ and $Z_2$. Fix an integer $y$. Take $P\in Z_1$. Let $\mathscr E=E_{d,g,y}(X)$ be a line bundle on $X$ with multi-degree $(\frac{d-g+y}{2},\frac{d-g-y}{2}+1)$ over $(Z_1,Z_2)$ if $2|d-g-y$ and multi-degree $(\frac{d-g+y+1}{2},\frac{d-g-y+1}{2})$ if $2\nmid d-g-y$. Then a simple torsion free rank one sheaf $I$ is $P$-quasistable if 
$$\chi(I_{Z_1})>\frac{d-g+y}{2}\mathrm{\ and\ }\chi(I_{Z_2})\geq\frac{d-g-y}{2}+1$$when $2| d-g-y$ and
$$\chi(I_{Z_1})>\frac{d-g+y+1}{2}\mathrm{\ and\ }\chi(I_{Z_2})\geq\frac{d-g-y+1}{2}$$
 when $2\nmid d-g-y$. 
 
 On the other hand, recall that in \cite{caporaso2010brill}, a multi-degree $(d_1,d_2)$ on a genus-$g$ binary curve is balanced if 
 $|d_1-d_2|\leq g+1$. We adapt this notion and say that $(d_1,d_2)$ is \textbf{$g$-balanced} if $|d_1-d_2|\leq g+1$. The advantage is that our adaption is independent of the underlying curve.  If $I$ is a line bundle on $X$ with multi-degree $(d_1,d_2)$ over $(Z_1,Z_2)$ and $I$ is $P$-quasistable with respect to $E_{d,g,0}(X)$, then $(d_1,d_2)$ is $g$-balanced.
  
 \end{ex}
 

Let $S$ be a locally Noetherian scheme over $k$. Let $\pi\colon \mathfrak X\rightarrow S$ be a projective, flat morphism whose fibers are curves of arithmetic genus $g$. Let $\mathcal I$ be an $S$-flat coherent sheaf on $\mathfrak X$. We say that $\mathcal I$ is \textbf{relatively torsion free} (resp. \textbf{rank one}, resp. \textbf{simple}) over $S$ if $\mathcal I(s)$ is torsion free (resp. rank one, resp. simple) for every (closed) point $s\in S$. We say that $\mathcal I$ has \textbf{relative degree $d$} over $S$ if $\mathcal I(s)$ has Euler characteristic $d-g+1$. Let $\mathcal E$ be a line bundle on $\mathfrak X$ of relative degree $d-g+1$ over $S$. We call such an $\mathcal E$ a \textbf{relative polarization} on $\mathfrak X$ over $S$. A relatively torsion free, rank one sheaf $\mathcal I$ on $\mathfrak X$ over $S$ is \textbf{relatively stable} (resp. \textbf{relatively semistable}) with respect to $\mathcal E$ over $S$ if $\mathcal I(s)$ is stable (resp. semistable) with respect to $\mathcal E(s)$ for all points $s\in S$.
Let $\sigma\colon S\rightarrow \mathfrak X$ be a section of $\pi$ through the smooth locus of $\mathfrak X$. 
A relatively torsion free, rank one sheaf $\mathcal I$ on $\mathfrak X$ over $S$ is \textbf{relatively $\sigma$-quasistable} with respect to $\mathcal E$ over $S$ if $\mathcal I(s)$ is $\sigma(s)$-quasistable with respect to $\mathcal E(s)$ for all points $s\in S$.

Let $\bm J^*$ denote the contravariant functor from the category of locally Noetherian $S$-schemes to sets, defined on an $S$-scheme $T$ by 
$$\bm J^*(T)=\{\mathrm{relatively\ simple,\ torsion\ free,\ rank\ one\ sheaves\ on}\ \mathfrak X\times_ST\ \mathrm{over}\ T\}/\sim,$$
where $``\sim"$ is the following equivalence relation:
$$\mathcal I_1\sim\mathcal I_2\Longleftrightarrow\mathrm{There\ is\ an\ invertible\ sheaf\ }M\mathrm{\ on\ } T\mathrm{\ such\ that\ }\mathcal I_1\cong\mathcal I_2\otimes \pi^*(M).$$
Let $\bm J$ be the \'etale sheaf associated to $\bm J^*$. Then $\bm J$ is represented by an algebraic space $J$ by \cite{altman1980compactifying}. Fix a degree $d$ and a relative polarization $\mathcal E$ on $\mathfrak X$ of relative degree $d-g+1$. Fix a section $\sigma$ of $\pi$. Let $J^\sigma_\mathcal E$ be the subspace of $J$ parametrizing relatively simple torsion free rank one sheaves on $\mathfrak X$ over $S$ that are relatively $\sigma$-quasistable with respect to $\mathcal E$. By \cite[Proposition 27]{esteves2001compactifying} $J^\sigma_\mathcal E$ is an open subspace of $J$.

\begin{thm}\label{section quasistable and properness}
(\cite[Theorem A (3)]{esteves2001compactifying}) $J^\sigma_\mathcal E$ is proper over $S$.
\end{thm}


We will use $J^\sigma_{\mathcal E}$ as the  compactified relative Jacobian in the proof of our main theorem, and vary the polarization $\mathcal E$ for different situations.

\subsection{Multi-vanishing sequence and ramification.}\label{subsec:Multi-vanishing}
Let $X$ be a curve of genus $g$. Fix integers $r,d>0$ and non-singular points $P_1,...,P_t$ of $X$. Let $D_\bullet=(D_0,D_1,...)$ be an increasing sequence, possibly infinite, of effective divisors supported on $\{P_1,..., P_t\}$.

\begin{defn}\label{defn: generalized linear series}
A \textbf{generalized linear series} of degree $d$ and rank $r$ on $X$, denoted by a ``generalized $\mathfrak g^r_d$" on $X$, is a pair $(I,V)$ consisting of a simple torsion free sheaf $I$ of rank one on $X$ such that $\chi (I)=d-g+1$ and a subspace $V\subset H^0(X,I)$ of dimension $r+1$.
\end{defn}

We say that a generalized linear series $(I,V)$ is $P$-\textbf{quasistable} with respect to a non-singular point $P\in X$ and a polarization $\mathscr E$ on $X$ if $I$ is so. Recall the notion of multi-vanishing sequence in \cite{osserman2019limit}: 

\begin{defn}\label{multivanishing sequence}
Let $(I,V)$ be a generalized $\mathfrak g^r_d$ on $X$. Inserting $D_{-1}=0$ to $D_\bullet$ if necessary we may assume $D_0=0$. The \textbf{multi-vanishing order} of a section $s\in H^0(X,I)$ along $D_\bullet$ is the degree of $D_l$ where $l$ is the greatest number such that $s\in H^0(X,I(-D_l))$. The \textbf{multi-vanishing sequence} of $(I,V)$ along $D_\bullet$ is the sequence $$v_0\leq v_1\leq\cdots\leq v_r$$ where a value $a$ appears in the sequence $m$ times if for some $j$ we have $\deg D_j=a$ and $$ \left\{\begin{array}{ll}
  \dim V(-D_j)=m   & \mathrm{if\ } D_\bullet \mathrm{\ ends\ with\ } D_j,\\
   \dim(V(-D_j)/V(-D_{j+1}))=m   & \mathrm{\ otherwise.}
\end{array}
\right.$$
\end{defn}

\begin{ex}\label{ex:multi-vanishing sequence}
Let $X$ be the projective line, $I=\mathcal O(4)$ and take a linear series $V$ generated by $x,x^3$ and $x^4-x^2$. Denote $P_1=0$, $P_2=1$, and let $$D_\bullet=0, P_1+P_2, 2P_1+P_2,2P_1+2P_2,3P_1+2P_2.$$
Then the multi-vanishing order of $x$ (resp. $x^3$, resp. $x^4-x^2$) is 0 (resp.0, resp. 3), and the multi-vanishing sequence of $(I,V)$ is 0,2,3. 
\end{ex}

From now on, for simplicity let $X$ be a binary curve of genus $g$ with components $Z_1$ and $Z_2$, and take $P_i\in Z_i$ for $1\leq i\leq t=2$. The multi-vanishing sequence of a generalized linear series along $D_\bullet$ satisfies the condition of being an admissible sequence, as defined below: 
 
\begin{defn}\label{defn: admissible sequence}
A nondecreasing sequence $(a_j)_{0\leq j\leq r}$ is called an \textbf{admissible sequence} along $D_\bullet$ if for each $j$ we have $a_j=\deg D_{l_j}$ for some $l_j$, and, whenever $D_{l_j+1}$ is defined, the number of repetitions of $a_j$ is at most $\deg (D_{l_j+1}-D_{l_j})$. For a multidegree $\underline d=(d_1,d_2)$ over $(Z_1,Z_2)$, we say that an admissible sequence $a_\bullet$ along $D_\bullet$ is \textbf{$\underline d$-bounded} if $$D_{l_r}\leq (d_1+1)P_1+(d_2+1)P_2,$$ and the number of repetitions of $\deg D_{l_r}$ in $a_\bullet$ is at most $d_1+d_2+3-\deg D_{l_r}$.
\end{defn}

\begin{rem}\label{rem:multi-vanishing}
(1) Note that, for an admissible sequence $a_\bullet$ along $D_\bullet$, the generalized linear series $(I,V)$ has multi-vanishing sequence at least $a_\bullet$ along $D_\bullet$ if and only if $\dim V(-D_{l_j})\geq r+1-j$ for all $0\leq j\leq r$, where $l_j$ is the number such that $a_j=\deg D_{l_j}$, as in Definition \ref{defn: admissible sequence}. As a result, if $\widetilde D_\bullet$ is an extension of $D_\bullet$ and $a_\bullet$ is an admissible sequence along $\widetilde D_\bullet$, then $(I,V)$ has multi-vanishing sequence at least $a_\bullet$ along $D_\bullet$ if and only if it has multi-vanishing sequence at least $a_\bullet$ along $\widetilde D_\bullet$.

(2) Fix an admissible sequence $a_\bullet$ along $D_\bullet$. Extending $D_\bullet$ if necessary we may  assume that there exists a term in $D_\bullet$ which has degree strictly more than $a_r$. Set $a_{-1}=-1$.
Let $D'_\bullet$ be the sequence of effective divisors obtained from $D_\bullet$ as follows: for all $(l,j,c)$ such that $$a_{j-1}<\deg D_l=a_j=a_{j+1}=\cdots=a_{j+c}<a_{j+c+1}$$ or $$a_{j-1}<\deg D_l=a_j=a_{j+1}=\cdots=a_r \mathrm{\ and\ } j+c=r,$$ insert a sequence of divisors $D_{l,1},...,D_{l,c}$ to $D_\bullet$ such that 
$$D_l< D_{l,1}< D_{l,2}< \cdots < D_{l,c}<D_{l+1}
\mathrm{\ and\ }\deg D_{l,u}=\deg D_l+u \mathrm{\ for\ all\ } 1\leq u\leq c.$$
Accordingly, for all $(j,c)$ as above replacing $a_{j+u}$ by $a_j+u$ for $1\leq u\leq c$ we obtain an admissible  sequence $a'_\bullet$ along $D'_\bullet$. Let $D''_\bullet$ be obtained from $D'_\bullet$ by removing divisors whose degree does not appear in $a'_\bullet$. Then $(I,V)$ has multi-vanishing sequence at least $a_\bullet$ along $D_\bullet$ if and only if it has multi-vanishing sequence at least $a'_\bullet$ along $D'_\bullet$, which is equivalent to having multi-vanishing sequence at least $a'_\bullet$ along $D''_\bullet$. Note that $D''_\bullet=(D''_0,...,D''_r)$ and $a'_j=\deg D''_j$ for $0\leq j\leq r$. Moreover, if $a_\bullet$ is $\underline d$-bounded, we can also pick $D''_\bullet$ such that  $D''_\bullet\leq (d_1+1)P_1+(d_2+1)P_2$. 
\end{rem}

\subsection{A family of moduli schemes of generalized linear series with ramifications}
\label{subsec: family}

Consider the family $\pi\colon \mathfrak X=X\times \mathbb A^1_k\rightarrow \mathbb A^1_k$. Fix a polarization $\mathcal E$ of relative multi-degree $(e_1,e_2)$ and a section $\sigma$ of $\pi$ through the smooth locus. It follows from \cite[Theorem B]{esteves2001compactifying} that $J^\sigma_{\mathfrak X,\mathcal E}=J^\sigma_\mathcal E$ is a scheme. Let $\mathcal L^\sigma_{\mathfrak X,\mathcal E}$ be the universal sheaf over $\mathfrak J^{\sigma}_{\mathfrak X,\mathcal E}=J^{\sigma}_{\mathfrak X,\mathcal E}\times_{\mathbb A^1_k}\mathfrak X$. Let $\widetilde \pi\colon \mathfrak X\rightarrow X$ and $p_1\colon \mathfrak J^\sigma_{\mathfrak X,\mathcal E}\rightarrow J^\sigma_{\mathfrak X,\mathcal E}$ and $p_2\colon \mathfrak J^\sigma_{\mathfrak X,\mathcal E}\rightarrow\mathfrak X$ be the natural maps as in the diagram below. For $1\leq i\leq 2$ let $D_i$ be an effective divisor on $Z_i\backslash (Z_1\cap Z_2)$ of degree at least $g+1-e_i$, and $\mathcal D'$ be the divisor $p^*_2(\widetilde\pi^*(D_1+D_2))$ on $\mathfrak J^\sigma_{\mathfrak X,\mathcal E}$. 

$$\begin{tikzcd}  \mathfrak J^\sigma_{\mathfrak X,\mathcal E}\rar{p_2}\dar{p_1} &\mathfrak X\dar{\pi }\rar{\widetilde\pi}&X\\ J^\sigma_{\mathfrak X,\mathcal E}\rar{} &\mathbb A^1_k
\end{tikzcd}$$


Note that the restriction of $\mathcal L^\sigma_{\mathfrak X,\mathcal E}(\mathcal D')$ at each fiber of $\mathfrak J^\sigma_{\mathfrak X,\mathcal E}$ over $J^\sigma_{\mathfrak X,\mathcal E}$ has a $(d-g+1+\deg(D_1+D_2))$-dimensional space of global sections and trivial first cohomology group. Indeed, for each torsion free sheaf $I$ in $J^\sigma_{\mathfrak X,\mathcal E}$ we have 
$$I\otimes \mathcal O_X(D_1+D_2)=n_{I*}(L_I\otimes \mathcal O_{X_I}(D_1+D_2)),$$
where $n_I$, $X_I$, and $L_I$ are defined in \S\ref{subsec:compactified}. Since $n_I$ is an affine morphism, by \cite[Exercise 3.8.2]{hartshorne1977algebraic} we have $$h^i(I\otimes \mathcal O_X(D_1+D_2))=h^i(L_I\otimes \mathcal O_{X_I}(D_1+D_2))\mathrm{\ for\ } i=0,1.$$ As $I_{Z_i}=L_I|_{Z_i}$ for each $1\leq i\leq 2$, the line bundle $L_I\otimes \mathcal O_{X_I}(D_1+D_2)$ has multi-degree at least $(g,g)$ along $(Z_1,Z_2)$ by semi-stability of $I$. Therefore the conclusion follows from Riemann-Roch theorem.

By \cite[\S 0.5]{mumford1994geometric} the pushforward $p_{1*}\mathcal L^\sigma_{\mathfrak X,\mathcal E}(\mathcal D')$ is a vector bundle of rank $d-g+1+\deg(D_1+D_2)$ on $J^\sigma_{\mathfrak X,\mathcal E}$ whose fibers at a point are identified with the global sections of the restriction of $\mathcal L^\sigma_{\mathfrak X,\mathcal E}(\mathcal D')$ to the fiber of $\mathfrak J^\sigma_{\mathfrak X,\mathcal E}$ at that point.

\begin{defn}\label{family of bn loci}
The family $\overline W^r_d(\mathfrak X/\mathbb A^1_k)$ of Brill-Noether loci over $\mathbb A^1_k$ is the locus on $J^\sigma_{\mathfrak X,\mathcal E}$ on which the map 
$$p_{1*}\mathcal L^\sigma_{\mathfrak X,\mathcal E}(\mathcal D')\rightarrow
p_{1*}\mathcal L^\sigma_{\mathfrak X,\mathcal E}(\mathcal D')|_{\mathcal D'}$$ has kernel of dimension at least $r+1$.
Here $\mathcal L^\sigma_{\mathfrak X,\mathcal E}(\mathcal D')|_{\mathcal D'}$ is the restriction of $\mathcal L^\sigma_{\mathfrak X,\mathcal E}(\mathcal D')$ on $\mathcal D'$:
\end{defn}

In the definition above, $\mathcal L^\sigma_{\mathfrak X,\mathcal E}$ is not a line bundle over $\mathfrak J^\sigma_{\mathfrak X,\mathcal E}$, but we still have the exact sequence 
$$0\rightarrow \mathcal L^\sigma_{\mathfrak X,\mathcal E}\rightarrow \mathcal L^\sigma_{\mathfrak X,\mathcal E}(\mathcal D')\rightarrow\mathcal L^\sigma_{\mathfrak X,\mathcal E}(\mathcal D')|_{\mathcal D'}\rightarrow 0,$$
since $\mathcal L^\sigma_{\mathfrak X,\mathcal E}$ is locally free at points where $\mathcal O_{\mathcal D'}$ is nonzero. Hence the kernel of the map in Definition \ref{family of bn loci} at a point of $J^\sigma_{\mathfrak X,\mathcal E}$ is identified with the space of global sections of the corresponding torsion free sheaf on $X$.

\begin{defn}\label{family of linear series}
Let $p\colon Gr(r+1,p_{1*}\mathcal L^\sigma_{\mathfrak X,\mathcal E}(\mathcal D'))\rightarrow J^\sigma_{\mathfrak X,\mathcal E}$ be the relative Grassmannian over $J^\sigma_{\mathfrak X,\mathcal E}$ and $\mathscr V\rightarrow p^*p_{1*}\mathcal L^\sigma_{\mathfrak X,\mathcal E}(\mathcal D')$ be the tautological subbundle. The family $\overline G^r_d(\mathfrak X/\mathbb A^1_k,\mathcal E,\sigma)$ of moduli spaces of generalized linear series on $X$ over $\mathbb A^1_k$ is the locus on $Gr(r+1,p_{1*}\mathcal L^\sigma_{\mathfrak X,\mathcal E}(\mathcal D'))$ where the composed map 
$$\mathscr V\rightarrow p^*p_{1*}\mathcal L^\sigma_{\mathfrak X,\mathcal E}(\mathcal D')\rightarrow p^*p_{1*}\mathcal L^\sigma_{\mathfrak X,\mathcal E}(\mathcal D')|_{\mathcal D'}$$ 
is zero.
\end{defn}

Let $\mathcal D$ be a divisor on $\mathfrak X$ supported on the smooth locus over $\mathbb A^1_k$. Suppose $\mathcal O_\mathcal D=\mathcal O_\mathfrak X/\mathcal O_\mathfrak X(-\mathcal D)$ is flat over $\mathbb A^1_k$. Then we can construct a family $\overline G^r_d(\mathfrak X/\mathbb A^1_k,\mathcal E,\sigma,\mathcal D,b)$ of moduli space of generalized linear series on $X$ over $\mathbb A^1_k$ with given ramification (corresponding to $b$) at $\mathcal D$ as the locus on $\overline G^r_d(\mathfrak X/\mathbb A^1_k,\mathcal E,\sigma)$ on which the induced map 
$$\mathscr V\rightarrow p^*p_{1*}(\mathcal L^\sigma_{X,\mathcal E}|_{p^*_2(\mathcal D)})$$
has kernel of dimension at least $b$.

\begin{defn}\label{family of imposed ramification}
For $i=1,2$ let $\mathcal P_i\subset Z_i\times \mathbb A^1_k$ be two sections of $\pi$ supported on the smooth locus of 
$\mathfrak X$. Let $\mathcal D_\bullet=a^1_\bullet \mathcal P_1+a^2_\bullet\mathcal P_2$ be an increasing sequence of effective divisors and $a_\bullet$ an admissible sequence along $\mathcal D_\bullet$ (or along $(\mathcal D_\bullet)_z$ for each $z\in \mathbb A^1_k$). For each $0\leq j\leq r$ take $l_j$ such that $a_j=a^1_{l_j}+a^2_{l_j}=\deg \mathcal D_{l_j}$. The family $\overline G^r_d(\mathfrak X/\mathbb A^1_k,\mathcal E,\sigma;\mathcal D_\bullet, a_\bullet)$ of moduli spaces over $\mathbb A^1_k$ of generalized linear series on $X$ with ramification imposed by $(\mathcal D_\bullet, a_\bullet)$ 
is the intersection over all $0\leq j\leq r$ of $\overline G^r_d(\mathfrak X/\mathbb A^1_k,\mathcal E,\sigma,\mathcal D_{l_j},r+1-j)$ in $\overline G^r_d(\mathfrak X/\mathbb A^1_k,\mathcal E,\sigma)$. 
\end{defn}

By \cite[\S0.5]{mumford1994geometric} our constructions above is compatible with base extension $T\rightarrow \mathbb A^1_k$. In particular, we have the following proposition:

\begin{prop}\label{prop:degeneration}
The spaces $\overline W^r_d(\mathfrak X/\mathbb A^1_k,\mathcal E,\sigma)$, $\overline G^r_d(\mathfrak X/\mathbb A^1_k,\mathcal E,\sigma)$, $\overline G^r_d(\mathfrak X/\mathbb A^1_k,\mathcal E,\sigma,\mathcal D,b)$ and \\
$\overline G^r_d(\mathfrak X/\mathbb A^1_k,\mathcal E,\sigma;\mathcal D_\bullet, a_\bullet)$ are all proper over $\mathbb A^1_k$. Moreover, we have

(1) The fiber of $\overline W^r_d(\mathfrak X/\mathbb A^1_k,\mathcal E,\sigma)$ at $z\in \mathbb A^1_k$ parameterizes $\sigma(z)$-stable (with respect to $\mathcal E_z$) simple torsion free sheaves of rank one on $X$ which have global sections of dimension at least $r+1$.

(2) The fiber of $\overline G^r_d(\mathfrak X/\mathbb A^1_k,\mathcal E,\sigma)$ at $z\in \mathbb A^1_k$ parameterizes $\sigma(z)$-quasistable generalized $\mathfrak g^r_d$s with respect to $\mathcal E_z$ on $X$.

(3) The fiber of $\overline G^r_d(\mathfrak X/\mathbb A^1_k,\mathcal E,\sigma,\mathcal D,b)$ at $z\in\mathbb A^1_k$ parametrizes $\sigma(z)$-quasistable generalized $\mathfrak g^r_d$s $(I,V)$ with respect to $\mathcal E_z$ on $\mathfrak X_z=X$ such that $\dim V(-\mathcal D_z)\geq b$.

(4) The fiber of $\overline G^r_d(\mathfrak X/\mathbb A^1_k,\mathcal E,\sigma;\mathcal D_\bullet, a_\bullet)$ at $z\in \mathbb A^1_k$ parameterizes $\sigma(z)$-quasistable generalized $\mathfrak g^r_d$s with respect to $\mathcal E_z$ on $\mathfrak X_z=X$ whose multi-vanishing sequence along $(\mathcal D_\bullet)_z$ is at least $a_\bullet$.
\end{prop}
Since $\overline G^r_d(\mathfrak X/\mathbb A^1_k,\mathcal E,\sigma;\mathcal D_\bullet, a_\bullet)$ is a Schubert subscheme (cf. \cite[\S C.2]{ossermanlimit}) of $\overline G^r_d(\mathfrak X/\mathbb A^1_k,\mathcal E,\sigma)$, and $\overline G^r_d(\mathfrak X/\mathbb A^1_k,\mathcal E,\sigma)$ is a determinantal subscheme of $Gr(r+1,p_{1*}\mathcal L^\sigma_{\mathfrak X,\mathcal E}(\mathcal D'))$, the fiber at each $z\in \mathbb A^1_k$ of $\overline G^r_d(\mathfrak X/\mathbb A^1_k,\mathcal E,\sigma;\mathcal D_\bullet, a_\bullet)$ (resp. $\overline G^r_d(\mathfrak X/\mathbb A^1_k,\mathcal E,\sigma)$) is expected to have dimension 
$$g-(r+1)(g-d+r)-\bigg(\sum_{j=0}^{r}(a_j-j)+\sum_{l\geq 0}\binom{r_l}{2}\bigg)\ \ (\mathrm{resp.\ } \rho_{g,r,d}).$$Here $r_l$ is the number of repetitions of $\deg D_l$ in $a_\bullet$


\subsection{A decomposition of the space of generalized linear series.}\label{subsec: stratifying} Now for a line bundle $\mathscr E$ on $X$ and a non-singular point $P$ of $X$ we denote by $\overline W^r_d(X,\mathscr E,P)$ (resp. $\overline G^r_d(X,\mathscr E,P)$) the space parameterizing $P$-quasistable simple torsion free, rank one sheaves (resp. generalized $\mathfrak g^r_d$s) on $X$. For a sequence $D_0,...,D_r$ of increasing effective divisors supported on $\{P_1, P_2\}$, where $P_i$ is a non-singular point in $Z_i$, as in \S\ref{subsec:Multi-vanishing}, denote by $\overline G^r_d(X,\mathscr E,P; D_\bullet, a_\bullet)$ the space parametrizing $P$-quasistable generalized $\mathfrak g^r_d$s with multi-vanishing sequence at least $a_\bullet$ along $D_\bullet$.

On the other hand, fix a multi-degree $\underline d=(d_1,d_2)$ over $(Z_1,Z_2)$. By a $\mathfrak g^r_{\underline d}$ on $X$ we mean a linear series $(L,V)$ on $X$ such that $L$ has multi-degree $\underline d$ and $V$ has dimension $r+1$. We denote by $W^r_{\underline d}(X)$ (resp. $G^r_{\underline d}(X)$, resp. $G^r_{\underline d}(X;D_\bullet,a_\bullet)$) the space parametrizing line bundles of multi-degree $\underline d$ on $X$ whose space of global sections has dimension at least $r+1$ (resp. $\mathfrak g^r_{\underline d}$s on $X$, resp. $\mathfrak g^r_{\underline d}$s on $X$ whose multi-vanishing sequence along $D_\bullet$ is at least $a_\bullet$).

Moreover, note that the stability for a torsion free sheaf $I$ of rank one on $X$ with respect to $\mathscr E$ only depends on the multi-degree of $L_I$. Accordingly, we say that a multi-degree $(d_1,d_2)$ is \textbf{stable} (resp. \textbf{semistable}, resp. $P$-\textbf{quasistable}) if $d_i+1>\deg\mathscr E|_{Z_i}$ for each $i$ (resp. $d_i+1\geq \deg\mathscr E|_{Z_i}$ for each $i$, resp. $d_i+1>\deg\mathscr E|_{Z_i}$ if $P\in Z_i$ and $d_i+1\geq \deg\mathscr E|_{Z_i}$ otherwise) for convenience.
Hence $I$ is $P$-quasistable with respect to $\mathscr E$ if and only if the multi-degree of $L_I$ is $P$-quasistable with respect to $\mathscr E$. Note also that the global sections of $I$ are identified with the global sections of $L_I$. As a result, we have 

\begin{prop}\label{prop: stratification of linear series} 
Let $S$ be the set of singular points of $X$. For each subset $J\subsetneq  S$ let $X_J$ be the normalization of $X$ at $J$. Then
$$\overline W^r_d(X,\mathscr E,P)=\bigcup_{J\subsetneq S}\bigcup_{\underline d_J} W^r_{\underline d_J}(X_J),$$
and 
$$\overline G^r_d(X,\mathscr E,P)=\bigcup_{J\subsetneq S}\bigcup_{\underline d_J}  G^r_{\underline d_J}(X_J),$$
and 
$$\overline G^r_d(X,\mathscr E,P; D_\bullet, a_\bullet)=\bigcup_{J\subsetneq S}\bigcup_{\underline d_J} G^r_{\underline d_J}(X_J; D_\bullet, a_\bullet).$$
Here $\underline d_J$ runs over all multi-degrees $(d^1_J,d^2_J)$ which are $P$-quasistable with respect to $\mathscr E$ such that $d^1_J+d^2_J=d-|J|.$
\end{prop}

\section{The proof of the main theorem.}
\label{sec: main}
In this section we restate Theorem \ref{thm: introduction} and give a proof. Our proof proceeds by induction on genus. For the inductive step we reduce the genus by imposing an extra ramification condition. For convenience we specify the following notation:
\begin{nota}\label{nota: linear series}
Let $X$ be a binary curve of genus $g$ with components $Z_1$ and $Z_2$. Let $D_\bullet$ (resp. $E_\bullet$) be an increasing sequence of effective divisors supported on $\{P_1, P_2\}$ (resp $\{P'_1, P'_2\}$), where $P_i$ and $P'_i$ are non-singular points of $X$ in $Z_i$ and $P_i\neq P'_i$ for $1\leq i\leq 2$. Let $a_\bullet$ (resp. $a'_\bullet$) be an admissible sequence along $D_\bullet$ (resp. $E_\bullet$). Let $\underline d$ be a multi-degree over $(Z_1,Z_2)$. Recall that $G^r_{\underline d}(X)$ (resp. $G^r_{\underline d}(X;D_\bullet,a_\bullet)$) denotes the space of $\mathfrak g^r_{\underline d}$s on $X$ (resp. $\mathfrak g^r_{\underline d}$s on $X$ with multi-vanishing sequence at least $a_\bullet$ along $D_\bullet$). Let $G^r_{\underline d}(X;D_\bullet,a_\bullet)^\circ$ be the subspace of $G^r_{\underline d}(X;D_\bullet,a_\bullet)$ parametrizing pairs $(L,V)$ such that $\dim h^0(X,L)=r+1$. Let $$\partial G^r_{\underline d}(X;D_\bullet,a_\bullet)=G^r_{\underline d}(X;D_\bullet,a_\bullet)\backslash G^r_{\underline d}(X;D_\bullet,a_\bullet)^\circ.$$
Let also $G^r_{\underline d}(X;D_\bullet,a_\bullet)^{\circ\circ}$ be the subspace of $G^r_{\underline d}(X;D_\bullet,a_\bullet)^\circ$ parametrizing $\mathfrak g^r_{\underline d}$s whose multi-vanishing sequence along $D_\bullet$ is exactly $a_\bullet$.
Moreover, denote by 
$$G^r_{\underline d}(X;D_\bullet,a_\bullet;D'_\bullet,a'_\bullet)=G^r_{\underline d}(X;D_\bullet,a_\bullet)\times_{G^r_{\underline d}(X)}G^r_{\underline d}(X;D'_\bullet,a'_\bullet)$$
the space of $\mathfrak g^r_d$s on $X$ with multi-vanishing sequence at least $a_\bullet$ along $D_\bullet$ and $a'_\bullet$ along $D'_\bullet$.
\end{nota}

\begin{lem}\label{lem: d<g+r}
Let $(d_1,d_2)$ be a multi-degree over $(Z_1,Z_2)$ such that either of the following is true:

1) $(d_1,d_2)$ is $g$-balanced as in Example \ref{ex: polarization stableness};

2) $d_i\geq -1$ for each $i$;

3) $d_i\leq g+r$ for each $i$.

Suppose $d<g+r$ and $G^r_{\underline d}(X)$ is nonempty. Then $r\leq d_i\leq d-r$ for each $i$.

\end{lem}
\begin{proof}
Take $(L,V)\in G^r_{\underline d}(X)$. If $d_1\geq g$, then the restriction map $$\phi\colon H^0(X,L)\rightarrow H^0(Z_2,L|_{Z_2})$$ is surjective with kernel of dimension $d_1-g$. As a result $$h^0(X,L)=d_1-g+\max\{d_2+1,0\}\leq r,$$
which is impossible. Hence we conclude $d_1\leq g$, and similarly $d_2\leq g$. It follows that $\phi$ is injective. Thus $r+1\leq h^0(Z_2,L|_{Z_2})=d_2+1$. As a result $d_2\geq r$ and $d_1=d-d_2\leq d-r$, and vice versa.
\end{proof}

Let $X'$ be the normalization of $X$ at a node $Q$, whose preimage in $Z_i$ is $Q_i$, as in Figure \ref{fig: normalization at one point}. Let $E_0=0$ and $E_1=Q_1+Q_2$ be divisors on $X'$. Let $b_0=0$ and $b_j=2$ for all $1\leq j\leq r$. Denote $G^r_{\underline d}(X';D_\bullet,a_\bullet;E_\bullet,b_\bullet)$ by $G_{E_0}$.
By pulling back a $\mathfrak g^r_d$ on $X$ to $X'$ we obtain a natural map

$$\varphi\colon G^r_{\underline d}(X;D_\bullet,a_\bullet)\rightarrow G_{E_0}.$$

\begin{figure}[H]
\begin{tikzpicture}[scale=0.5]
    
    \draw[thick] plot [smooth,tension=1] coordinates{(1,5)(0.5,4)(-0.5,2)(0.5,0)(-0.5,-2)(0.5,-4)(1,-5)};
    \draw[blue,thick] plot [smooth,tension=1] coordinates{(-1,5)(-0.5,4)(0.5,2)(-0.5,0)(0.5,-2)(-0.5,-4)(-1,-5)};
    
    \draw[shift={(0,1)}][thick] plot [smooth,tension=1] coordinates{(12,4)(11.5,3)(10.5,1)(11.5,-1)(10.5,-3)(10,-4)(9,-6)};
    \draw[shift={(0,1)}][blue,thick] plot [smooth,tension=1] coordinates{(10,4)(10.5,3)(11.5,1)(10.5,-1)(11.5,-3)(12,-4)(13,-6)};
     \draw (-1.5,4.5) node[circle, fill=black, scale=0, label=above:{$X$}]{};
     \draw (-0.75,-4.45) node[circle, fill=black, scale=0.3, label=left:{$P_1$}]{};
      \draw (0.85,-4.7) node[circle, fill=black, scale=0.3, label=right:{$P_2$}]{};
       \draw (0,-3.3) node[circle, fill=black, scale=0.3, label=right:{$Q$}]{};
       
        \draw (10,5) node[circle, fill=black, scale=0, label=left:{$X'$}]{};
        \draw (12.75,-4.45) node[circle, fill=black, scale=0.3, label=right:{$P_1$}]{};
        \draw (12.15,-3.3) node[circle, fill=black, scale=0.3, label=right:{$Q_1$}]{};
        \draw (9.15,-4.7) node[circle, fill=black, scale=0.3, label=right:{$P_2$}]{};
        \draw (9.85,-3.3) node[circle, fill=black, scale=0.3, label=right:{$Q_2$}]{};
\end{tikzpicture}
\caption{}\label{fig: normalization at one point}
\end{figure}

Let $G_{E_1}$ be the subspace of $G_{E_0}$ consisting of $\mathfrak g^r_{\underline d}$s with base points $Q_1$ and $Q_2$. We have:
\begin{lem}\label{lemma: induction step}
The map $\varphi$ above is injective over $G_{E_0}\backslash G_{E_1}$ and has fiber of dimension at most one over $G_{E_1}$.

\end{lem}
\begin{proof}
Let $n_Q$ be the normalization map. We have a commutative diagram
$$\begin{tikzcd}  G^r_{\underline d}(X;D_\bullet,a_\bullet)\rar{\varphi}\dar & G_{E_0}\dar\\ \mathrm{Pic}(X)\rar{n_Q^*}&\mathrm{Pic}(X').
\end{tikzcd}
$$
Let $(L',V')$ be a $\mathfrak g^r_d$ on $X'$ and $L$ be a line bundle on $X$ such that $n_Q^*(L)=L'$. Since the induced map $H^0(X,L)\rightarrow H^0(X',L)$ is injective, there is at most one $(r+1)$-dimensional subspace $V$ of $H^0(X,L)$ that maps to $V'$. Therefore the fiber dimension of $\varphi$ at $G_{E_1}$ is at most one, since the fiber of $n_Q^*$ has dimension one. If $(L',V')\not\in G_{E_1}$, then we can take a section $s\in V'$ such that either $s(Q_1)\neq 0$ or $s(Q_2)\neq 0$. Then either there is no section of $L$ that restricts to $s$ or the glueing data of $L$ at $Q$ is determined by $s$. Hence the fiber of $\varphi$ at $(L',V')$ consists of at most one point.
\end{proof}

We now restate and prove the main theorem.
\begin{thm}\label{thm: main theorem}
Let $X$ be a general genus $g$ binary curve with components $Z_1$ and $Z_2$. Let $P_i\in Z_i$ be two general points. Let $\underline d=(d_1,d_2)$ be a multi-degree where $-1\leq d_i$ for each $i$ and $d_1+d_2=d$. Let $r\geq 0$ be an integer. Let $a_\bullet$ be an admissible sequence along an increasing sequence $D_\bullet$ of effective divisors supported on $\{P_1, P_2\}$. Suppose either (i) $ d_i\leq g$ for each $i$ or (ii) $a_\bullet$ is $\underline d$-bounded. Then the space $G^r_{\underline d}(X;D_\bullet,a_\bullet)$ of linear series of multi-degree $\underline d$ and rank $r$ which has multi-vanishing sequence at least $a_\bullet$ along $D_\bullet$ is either empty or of dimension 
$$\rho=g-(r+1)(g-d+r)-\bigg(\sum_{j=0}^{r}(a_j-j)+\sum_{l\geq 0}\binom{r_l}{2}\bigg).$$Here $r_l$ is the number of repetitions of $\deg D_l$ in $a_\bullet$. In particular, $G^r_{\underline d}(X)$ is either empty or of dimension $\rho_{g,r,d}$.
\end{thm}


\begin{proof} According to Remark \ref{rem:multi-vanishing} (2) we may assume $D_\bullet=(D_0,...,D_r)$, where $D_j=a^1_jP_1+a^2_jP_2$,  and $a_j=a^1_j+a^2_j=\deg D_j$ for each $0\leq j\leq r$. Hence $a^i_{j+1}\geq a^i_j\geq 0$ and $a_{j+1}>a_j$ for each $i$ and $j$. Moreover, we have either (i) $d_i\leq g$ for each $i$ or (ii) $a^i_r\leq d_i+1$ for each $i$, and the expected dimension $\rho$ becomes
$$\rho=g-(r+1)(g-d+r)-\sum_{j=0}^{r}(a_j-j).$$

Since $G^r_{\underline d}(X;D_\bullet,a_\bullet)$ is a Schubert subscheme of $G^r_{\underline d}(X)$, which is a determinantal subscheme of the relative Grassmannian over $\mathrm{Pic}^{\underline d}(X)$, it has dimension at least $\rho$ if nonempty. Hence it remains to show that $G^r_{\underline d}(X;D_\bullet,a_\bullet)$ has dimension at most $\rho$.

We prove this by induction on $g$. Since when $d_i\leq g$ for each $i$ and $a^1_r>d_1+1$ the space $G^r_{\underline d}(X;D_\bullet,a_\bullet)$ is empty, we may always assume $a^1_r\leq d_1+1$. Similarly assume $a^2_r\leq d_2+1$. For the $g=0$ case there is a unique line bundle $L_{\underline d}$ on $X$ with multi-degree $\underline d$. Since $h^0(X,L_{\underline d})=d+1$, the space  $G^r_{\underline d}(X)$ of linear series on $X$ is identified with the Grassmannian $Gr(r+1,d+1)$, and $G^r_{\underline d}(X;D_\bullet,a_\bullet)$ is a Schubert variety defined by the partial flag $$H^0(X,L_{\underline d}(-D_r))\subset H^0(X,L_{\underline d}(-D_{r-1}))\subset\cdots\subset H^0(X,L_{\underline d}(-D_0)).$$
Note that $H^0(X,L_{\underline d}(-D_j))$ has dimension exactly $d-a_j+1$. Therefore it is easy to verify that $G^r_{\underline d}(X;D_\bullet,a_\bullet)$ has dimension as expected.

Suppose the conclusion is true for $g-1$. 
Let $X'$, $Q$, $Q_i$ and $G_{E_j}$ be as above. By Lemma \ref{lemma: induction step}, it is enough to show that $G_{E_0}$ has dimension (at most) $\rho$, and $G_{E_1}$ has dimension at most $\rho-1$. If $a^i_r=d_i+1$ for each $i$ then $G^r_{\underline d}(X';D_\bullet,a_\bullet)$ is empty, so is $G_{E_1}$. If $a^1_r\leq d_1$, then $G^r_{( d_1-1,d_2)}(X';D_\bullet,a_\bullet)$ has dimension at most $\rho-1$ by induction. Since it contains $G_{E_1}$, the dimension of $G_{E_1}$ is at most $\rho-1$, and similarly for the case $a^2_r\leq d_2$. We next show that $G_{E_0}$ has dimension at most $\rho$.

Consider the trivial family $\pi\colon \mathfrak X'=X'\times \mathbb A_k^1\rightarrow \mathbb A_k^1$. Denote by $\widetilde \pi\colon \mathfrak X' \rightarrow X'$ the other projection. Denote by $\tilde g=g-1$.

1) Assume $d\leq g+r-2=\tilde g+r-1$. By Lemma \ref{lem: d<g+r} we may assume $r\leq d_i\leq \tilde g-1$ for $1\leq i\leq 2$. Take $\mathcal E=\widetilde \pi^*(\mathscr E)$, where $\mathscr E =E_{d,\tilde g,0}(X')$ is defined in Example \ref{ex: polarization stableness}. Let $\sigma=P\times\mathbb A^1_k$ where $P\in Z_1$. It is easy to verify that $G_{E_0}$ is contained in $\overline G^r_d(X',\mathscr E,P)$.

Consider the degeneration $Q_1\rightarrow P_1$ as $z\rightarrow 0$. Let $G$ be the limit of $G_{E_0}$ in $\overline G^r_d(\mathfrak X'/\mathbb A^1_k,\mathcal E,\sigma)$ at $z=0$. Take $(I_0,V_0)\in G$. For each $0\leq l\leq r$ we have 
$$\dim V(-a^1_lP_1-a^2_lP_2-Q_1-Q_2)\geq r-l\mathrm{\ for\ all\ } (I,V)\in G_{E_0}.$$ By Proposition \ref{prop:degeneration} (3) (restrict to an open subvariety of $\mathbb A^1_k$ if necessary to make sure the limit $Q_1\rightarrow P_1$ misses the singular points of each fiber of $\mathfrak X'$) , it follows that
$$\dim V_0(-(a^1_{l}+1)P_1-a^2_{l}P_2-Q_2)\geq r-l$$ for each $l$.

Let $m\in\{-1,0,1,...,r\}$ be the unique number such that $$\dim V_0(-(a^1_l+1)P_1-a^2_lP_2)\geq r-l+1$$ for all $0\leq l\leq m$ and  $$\dim V_0(-(a^1_{l}+1)P_1-a^2_{l}P_2)=r-l.$$ 
for $l=m+1$ (if $m=r$ then this condition is waived).

If $m=r$, then 
$$\dim V_0(-(a^1_l+1)P_1-a^2_lP_2)\geq r-l+1$$
for all $0\leq l\leq r$.

We next assume $m<r$. Then $$\dim V_0(-a^1_{m+1}P_1-a^2_{m+1}P_2)=r-m$$ 
and 
$$ \dim V_0(-(a^1_{m+1}+1)P_1-a^2_{m+1}P_2)=r-m-1=\dim  V_0(-(a^1_{m+1}+1)P_1-a^2_{m+1}P_2-Q_2).$$ We thus have 
$$V_0(-(a^1_{m+1}+1)P_1-a^2_{m+1}P_2)= V_0(-(a^1_{m+1}+1)P_1-a^2_{m+1}P_2-Q_2).$$
Namely all sections in $V_0(-(a^1_{m+1}+1)P_1-a^2_{m+1}P_2)$ vanish at $Q_2$.



Let $n$ be the smallest number such that $a^1_n\geq a^1_{m+1}+1$. Then for all $l\geq n$ we have $$V_0(-a^1_lP_1-a^2_lP_2)\subset V_0(-(a^1_{m+1}+1)P_1-a^2_{m+1}P_2).$$
Hence all sections of $V_0(-a^1_lP_1-a^2_lP_2)$ vanish at $Q_2$. As a result 
$$\dim V_0(-a^1_lP_1-a^2_lP_2-Q_2)=\dim V_0(-a^1_lP_1-a^2_lP_2)\geq r-l+1.$$

Let $\mathfrak X'_0$ be the fiber of $\mathfrak X'$ at $z=0$. It follows that $G$ is contained in a union over all possible $(m,n)$ of subspaces $G_{m,n}$ of $\overline G^r_d(\mathfrak X'_0,\mathscr E,P)$ consisting of pairs  $(I_0,V_0)$ satisfying the flowing conditions: 

$$
\left\{
\arraycolsep=4pt\def\arraystretch{1.5}
\begin{array}{llll}
\dim V_0(-a^1_lP_1-a^2_lP_2-Q_2)\geq r-l+1,& l\geq n\\
\dim V_0(-(a^1_l+1)P_1-a^2_lP_2)\geq r-l+1,& l\leq m\\
\dim V_0(-a^1_lP_1-a^2_lP_2)\geq r-l+1,&0\leq l\leq r\\
\dim V_0(-(a^1_{l}+1)P_1-a^2_{l}P_2-Q_2)\geq r-l,&0\leq l\leq r.
\end{array}
\right.
$$

Now let $Q_2\rightarrow P_2$, then the conditions above become (again, by Proposition \ref{prop:degeneration} (3))

$$
\left\{
\arraycolsep=4pt\def\arraystretch{1.5}
\begin{array}{llll}
\dim V_0(-a^1_lP_1-(a^2_l+1)P_2)\geq r-l+1,& l\geq n\\
\dim V_0(-(a^1_l+1)P_1-a^2_lP_2)\geq r-l+1,& l\leq m\\
\dim V_0(-a^1_lP_1-a^2_lP_2)\geq r-l+1,&0\leq l\leq r\\
\dim V_0(-(a^1_{l}+1)P_1-(a^2_{l}+1)P_2)\geq r-l,&0\leq l\leq r.
\end{array}
\right.
$$

Let $h$ be the smallest number such that $a^1_h=a^1_{m+1}$ and $u\in\{h-1,h,...,n-1\}$ the unique number such that 
$$\dim V_0(-a^1_lP_1-(a^2_l+1)P_2)\geq r-l+1$$ for all $h\leq l\leq u$ and  $$\dim V_0(-a^1_lP_1-(a^2_l+1)P_2))=r-l=\dim V_0(-(a^1_l+1)P_1-(a^2_l+1)P_2).$$ 
for $l=u+1$. 
Similarly as above, we have 
$$ V_0(-a^1_{u+1}P_1-(a^2_{u+1}+1)P_2))= V_0(-(a^1_{u+1}+1)P_1-(a^2_{u+1}+1)P_2).$$
Now for all $u+1<l\leq n-1$ we have 
$$V_0(-a^1_lP_1-a^2_lP_2)\subset V_0(-a^1_{u+1}P_1-(a^2_{u+1}+1)P_2),$$
thus 
$$V_0(-a^1_lP_1-a^2_lP_2)=V_0(-(a^1_l+1)P_1-a^2_lP_2).$$

It follows that the limit of $G_{m,n}$ in $\overline G^r_d(\mathfrak X'/\mathbb A^1_k,\mathcal E,\sigma)$ as $Q_2\rightarrow P_2$ is contained in the union over all possible $(u,h)$ of
$\overline G^r_{d}( X',\mathscr E,P;D'_\bullet,a'_\bullet)$, where 
$$(a^1_j)'=\left\{ 
\arraycolsep=4pt\def\arraystretch{1.5}
\begin{array}{ll}a^1_j+1,& j< h\mathrm{\ or\ } u+1<j<n\\
a^1_j,& h\leq j\leq u+1\mathrm{\ or\ } j\geq n
\end{array}\right.$$
and
$$
(a^2_j)'=\left\{ 
\arraycolsep=1.6pt\def\arraystretch{1.5}
\begin{array}{ll}a^2_j,& j<h\mathrm{\ or\ } u+1\leq j\leq n-1\\
a^2_j+1,& h\leq j\leq u \mathrm{\ or\ }j\geq n
\end{array}\right.$$
and 
$D'_j=(a^1_j)'P_1+(a^2_j)'P_2$ and $a'_j=\deg D'_j$ (it is possible that $D'_u=D'_{u+1}$, in this case we replace $(a^1_{u+1})' $ with $(a^1_{u+1})'+1 $) for $0\leq j\leq r$. Note that we have $a'_j=a_j+1$ unless $j=u+1$ and $D'_u\neq D'_{u+1}$, in which case $a'_j=a_j$. 

By Proposition \ref{prop: stratification of linear series}, we have 
$$\overline G^r_d(X',\mathscr E,P; D'_\bullet, a'_\bullet)=\bigcup_{J\subsetneq S}\bigcup_{\underline d_J} G^r_{\underline d_J}(X'_J; D'_\bullet, a'_\bullet).$$
where $S$ is the set of nodes of $X'$, and $X'_J$ is the normalization of $X'$ at $J$, and $\underline d_J$ runs over all $P$-quasistable multi-degrees $(d^1_J,d^2_J)$ with respect to $\mathscr E$ such that $d_J:=d^1_J+d^2_J=d-|J|$. Note that $X'_J$ has genus $g_J=\tilde g-|J|$. Straightforward calculation shows that $(d^1_J,d^2_J)$ is $g_J$-balanced. Hence by Lemma \ref{lem: d<g+r} either $G^r_{\underline d_J}(X'_J; D'_\bullet, a'_\bullet)$ is empty or $r\leq d^i_J\leq g_J$ for $1\leq i\leq 2$. By induction 
$G^r_{\underline d_J}(X'_J; D'_\bullet, a'_\bullet)$ has dimension at most $$ \rho_J:=g_J-(r+1)(g_J-d_J+r)-\sum_{j=0}^{r}(a'_j-j)\leq \rho.$$ Hence $G_{m,n}$ has dimension at most $\rho$, and so do $G$ and $G_{E_0}$.

2) Let $d=g+r-1=\tilde g+r$. In this case we can not apply Lemma \ref{lem: d<g+r} for $X'$ and $\tilde g$ (or for any partial normalization $X'_J$ of $X'$ and $g_J$), as in part 1). Nevertheless, we can still apply this lemma for $X$ and $g$ and assume $r\leq d_i\leq \tilde g$ for $1\leq i\leq 2$.  We may assume $a^i_r\leq d_i$ for each $i$, since otherwise $G^r_{\underline d}(X';D_\bullet,a_\bullet)$ is empty. Without loss of generality, suppose $a^1_r\geq a^2_r$, and let $y=a^1_r-a^2_r\geq 0$. Take a line bundle $\mathscr E_y=E_{d,\tilde g,y}(X')$ on $X'$ as in Example \ref{ex: polarization stableness} and let $\mathcal E_y=\widetilde \pi^*(\mathscr E_y)$. As in part 1) let $\sigma=P\times \mathbb A^1_k$ where $P\in Z_1$. We first verify that $G^r_{\underline d}(X')$ is contained in $\overline G^r_d(X',\mathscr E_y,P)$.
Indeed, we have

$$d_1+1\geq a^1_r+1=\frac{1}{2}(a_r+y)+1\geq \frac{r+y}{2}+1>\frac{d-\tilde g+y+1}{2}$$
and $$d_2+1=d-d_1+1\geq d-\tilde g+1\geq \frac{d-\tilde g-y}{2}+1.$$
Hence $G^r_{\underline d}(X')\subset\overline G^r_d(X',\mathscr E_y,P)$.


Now let $Q_1\rightarrow P_1$ and $Q_2\rightarrow P_2$. By the argument in part 1), the limit of 
$G_{E_0}$
is contained in a union of subspaces of $\overline G^r_d(X',\mathscr E_y,P)$ of the form $$\overline G^r_d(X',\mathscr E_y,P; D'_\bullet, a'_\bullet)=\bigcup_{J\subsetneq S}\bigcup_{\underline d_J} G^r_{\underline d_J}(X'_J; D'_\bullet, a'_\bullet),$$
where $D'_\bullet$ and $a'_\bullet$ are as above and $\underline d_J$ runs over all $P$-quasistable multi-degrees with respect to $\mathscr E_y$. Note that $$a^i_j\leq (a^i_j)'\leq a^i_j+1 \mathrm{\ for\ each\ } i\mathrm{\ and\ }j.$$
It follows that $y-1\leq (a^1_r)'-(a^2_r)'\leq y+1$.

For each $\underline d_J=(d^1_J,d^2_J)$, let $d_J=d^1_J+d^2_J$. If $(a^1_r)'> d^1_J+1$, then $d^2_J=d_J-d^1_J\leq d_J-\frac{d-\tilde g+y}{2}=d_J-(d-\tilde g)+1+\frac{d-\tilde g+y}{2}-(y+1)\leq g_J+1+d^1_J-(y+1)< g_J+1+(a^1_r)'-(y+1)\leq g_J+1+(a^2_r)'.$ Hence there is no nonzero section $s$ of a line bundle on $X'_J$ with multi-degree $\underline d_J$ such that $s$ has multi-vanishing order at least $a'_r$ along $D'_\bullet$. Thus $G^r_{\underline d_J}(X'_J; D'_\bullet, a'_\bullet)$ is empty. Similarly, if $(a^2_r)'>d^2_J+1$, then $d^1_J=d_J-d^2_J\leq d_J-\frac{d-\tilde g-y-1}{2}<d_J-(d-\tilde g)+1+\frac{d-\tilde g-y-1}{2}+2+(y-1)\leq g_J+1+d^2_J+2+(y-1)\leq g_J+1+(a^2_r)'+(y-1)\leq g_J+1+(a^1_r)'.$ Hence $G^r_{\underline d_J}(X'_J; D'_\bullet, a'_\bullet)$ is empty.

We now suppose $(a^i_r)'\leq d^i_J+1$ for each $i$, in other words $a'_\bullet$ is $\underline d_J$ bounded. 
By induction each $G^r_{\underline d_J}(X'_J; D'_\bullet, a'_\bullet)$ has dimension at most $\rho_J$, which is defined in part 1). Hence the space $\overline G^r_d(X',\mathscr E_y,P; D'_\bullet, a'_\bullet)$ has dimension at most $\rho$, and so does
$G_{E_0}$.

3) Let $d=g+r$. 
Recall that we assumed $a^i_r\leq d_i+1$ for each $i$ in the beginning of the proof. We extend $D_\bullet$ to an infinite sequence $\widehat D_\bullet$ such that $\deg \widehat D_j=j$ for all $j\geq 0$ and $\widehat D_{d+2}=(d_1+1)P_1+(d_2+1)P_2$. Then 
$$G^r_{\underline d}( X;\widehat D_\bullet,a_\bullet)=G^r_{\underline d}( X; D_\bullet,a_\bullet)$$
by Remark \ref{rem:multi-vanishing} (1).

3.1) We first show that $G^r_{\underline d}( X;\widehat D_\bullet,a_\bullet)^\circ$ has dimension at most $\rho$. If $a_r=r$ then $\rho=g$ and the conclusion is obvious. We may assume $a_r>r$. Let $\lambda$ be the smallest number such that $a_\lambda>\lambda$. 
Consider an injective map 
$$G^r_{\underline d}( X;\widehat D_\bullet,a_\bullet)^{\circ\circ}\rightarrow G^{r-\lambda}_{\underline d'}(  X;D''_\bullet,a''_\bullet)^{\circ\circ}$$
where $(L,V)$ is mapped to $(L(- D_\lambda),V(- D_\lambda))$. Here $d'_i=d_i-a^i_\lambda$ and $a''_j=a_{j+\lambda}-a_\lambda$ and $D''_j=D_{j+\lambda}-D_\lambda$ for $1\leq i\leq 2$ and $0\leq j\leq r-\lambda$, and $\underline d'=(d'_1,d'_2)$.  
Note that $ G^{r-\lambda}_{\underline d'}(  X;D''_\bullet,a''_\bullet)$ is expected to have  dimension $\rho$.
Note also that $d':=d'_1+d'_2< g+r-\lambda$, and $$d'_1=d_1-a^1_\lambda \leq  d_1-\lambda-1+a^2_\lambda\leq d_1-\lambda+d_2=g+r-\lambda,$$ and similarly $d'_2\leq g+r-\lambda$. By Lemma \ref{lem: d<g+r}, either $0\leq d'_i\leq g$ for $1\leq i\leq 2$ or $ G^{r-\lambda}_{\underline d'}(  X;D''_\bullet,a''_\bullet)$ is empty. Hence by part 1) and 2) $G^{r-\lambda}_{\underline d'}(  X;D''_\bullet,a''_\bullet)^{\circ\circ}$ is either empty or of dimension $\rho$, and so is $G^r_{\underline d}( X;\widehat D_\bullet,a_\bullet)^{\circ\circ}$. 

Now let $\tilde a_\bullet \geq a_\bullet$ be another admissible sequence along $\widehat D_\bullet$. Note that $\tilde a_j=\deg \widehat D_{\tilde a_j}$ for all $j$. Let $\widehat D_{\tilde a_r}=\tilde a^1_rP_1+\tilde a^2_rP_2.$ By construction of $\widehat D_\bullet$, if $\tilde a^1_r>d_1+1$ then $\tilde a^2_r\geq d_2+1$. Hence $G^r_{\underline d}( X;\widehat D_\bullet,\tilde a_\bullet)^{\circ\circ}$ is empty. We may assume $\tilde a^i_r\leq d_i+1$ for each $i$. Then same argument as above shows that $G^r_{\underline d}( X;\widehat D_\bullet,\tilde a_\bullet)^{\circ\circ}$ is either empty or of expected dimension. Moreover, there are only finitely many $\tilde a_\bullet\geq a_\bullet$ such that $G^r_{\underline d}( X;\widehat D_\bullet,\tilde a_\bullet)^{\circ\circ}$ is nonempty. Now it follows from  
$$G^r_{\underline d}( X;\widehat D_\bullet,a_\bullet)^\circ=\bigcup_{\tilde a_\bullet\geq a_\bullet}G^r_{\underline d}( X;\widehat D_\bullet,\tilde a_\bullet)^{\circ\circ}$$
that $G^r_{\underline d}( X;\widehat D_\bullet,a_\bullet)^\circ$ is either empty or of expected dimension.

3.2) We show that $\partial G^r_{\underline d}( X;\widehat D_\bullet,a_\bullet)$ has dimension less than $\rho$. There is a map 
$$\partial G^r_{\underline d}( X;\widehat D_\bullet,a_\bullet)\rightarrow \bigcup_{\tilde r\geq r+1}\bigcup_{\hat a_\bullet}G^{\tilde r}_{\underline d}( X;\widehat D_\bullet,\hat a_\bullet)^{\circ\circ}$$ defined by $(L,V)\mapsto (L,\Gamma( X,L))$. 
Here the second union is over all $\hat a_\bullet$ such that there is a sequence $0\leq l_0<l_1<\cdots <l_r\leq \tilde r$ such that $\hat a_{l_j}\geq a_j$ for all $0\leq j\leq r$. Since $d<g+\tilde r$ for all $\tilde r\geq r+1$, every $G_{\tilde r,\hat a_\bullet}:=G^{\tilde r}_{\underline d}( X;\widehat D_\bullet,\hat a_\bullet)^{\circ\circ}$ is either empty or has expected dimension by part 1) and part 2), and the union above is a finite union. On the other hand, consider a tuple $(\tilde r,\hat a_\bullet)$ and take $\{l_j\}_{0\leq j\leq r}$ as above  which minimizes the sum $\sum_{0\leq j\leq r}l_j$. Denote by $\lambda_j=\hat a_{l_j}$ for $0\leq j\leq r$. Then for all $(\widetilde {L},\widetilde V)\in G_{\tilde r,\hat a_\bullet}$, the preimage of $(\widetilde {L},\widetilde V)$ is a Schubert variety in the Grassmannian Gr$(r+1,\widetilde V)$ consisting of spaces $V\subset \widetilde V$ such that 
$$\dim(V\cap \widetilde V(-\widehat D_{\lambda_j}))\geq r+1-j \mathrm{\ for\ } 0\leq j\leq r, $$
which has dimension  
$$(\tilde r-r)(r+1)-\sum_{j=0}^{ r}(l_j-j).$$
It follows that the preimage of $G_{\tilde r,\hat a_\bullet}$ has dimension 

\begin{equation*}
\begin{split}
& g-(\tilde r+1)(g-d+\tilde r)-\sum_{j=0}^{\tilde r}(\hat a_j-j)+(\tilde r-r)(r+1)-\sum_{j=0}^{ r}(l_j-j)\\
\leq &g-(\tilde r+1)(g-d+\tilde r)-\sum_{j=0}^{ r}(\hat a_{l_j}-l_j)+(\tilde r-r)(r+1)-\sum_{j=0}^{ r}(l_j-j)\\
=&\rho-(\tilde r-r)(g-d+\tilde r)-\sum_{j=0}^r(\hat a_{l_j}-a_j)\\
\leq &\rho-(\tilde r-r)(g-d+\tilde r)
\\
<&\rho.
\end{split}
\end{equation*}
Hence $\partial G^r_{\underline d}( X;\widehat D_\bullet,a_\bullet)$ has dimension less than $\rho$.

It follows that $G^r_{\underline d}( X;\widehat D_\bullet,a_\bullet)$ is either empty or of dimension $\rho$, and so is $G^r_{\underline d}( X; D_\bullet,a_\bullet)$.

4) Assume $d>g+r$. In this case $G^r_{\underline d}( X;\widehat D_\bullet,a_\bullet)^\circ$ and $G^{\tilde r}_{\underline d}( X;\widehat D_\bullet,\hat a_\bullet)^{\circ\circ}$ are empty for all $r<\tilde r<d-g$ by Riemann-Roch theorem. Here $\widehat D_\bullet$ and $\hat a_\bullet$ are the same as in part 3.2). Then by 
the same argument in part 3.2) and the conclusion of part 3) $\partial G^r_{\underline d}( X;\widehat  D_\bullet,a_\bullet)$ has dimension at most 
$$\max_{\tilde r\geq d-g}\{\rho-(\tilde r-r)(g-d+\tilde r)\}=\rho.$$

Hence $G^r_{\underline d}( X;\widehat D_\bullet,a_\bullet)$ is either empty or
of dimension at most $\rho$, and so is $G^r_{\underline d}( X;D_\bullet,a_\bullet).$
\end{proof}

The following example shows that our restrictions on the multi-degree and the ramification condition are necessary:

\begin{ex}\label{ex: not bn general} Let $X$ be an arbitrary binary curve.

(1) Suppose $d_1\geq g$ and $d_2\leq -2$, and $d\geq g+r-1$. As discussed in Lemma \ref{lem: d<g+r}, each line bundle of multi-degree $\underline d=(d_1,d_2)$ has $(d_1-g)$-dimensional space of global sections. Hence $$\dim G^r_{\underline d}(X)=g+(r+1)(d_1-g-r-1)>\rho(g,r,d)$$
as $d_1-1>d$. 

(2) Suppose $d_1\geq g+1$ and $d\geq g+r$. Let $D_j=a^1_jP_1+a^2_j P_2$ and $a_j=\deg D_j$ for $0\leq j\leq r$, where $a^1_j=0$ for $0\leq j\leq r$ and $a^2_j=j$ for $0\leq j\leq r-1$. Suppose $a^2_r>\rho_{g,r,d}+r$. Then $a_\bullet$ is not $\underline d$-bounded, and the expected dimension of $G^r_{\underline d}(X;D_\bullet,a_\bullet)$ is negative. However, each line bundle $L$ of multi-degree $\underline d$ contains a non-zero global section $s$ vanishing on $Z_2$. Hence any $(r+1)$-dimensional subspace of $H^0(X,L)$ containing $s$, which exists by Riemann-Roch theorem, has multi-vanishing sequence at least $a_\bullet$ along $D_\bullet$.   
\end{ex}

We end this paper with explaining the obstruction of applying the idea in \cite{griffiths1980variety} directly to our case and the connection between this idea and our proof of Theorem \ref{thm: main theorem}.

\begin{rem}\label{rem:issue}
Suppose $X$ is obtained from glueing $Q^1_j\in Z_1$ to $Q^2_j\in Z_2$ for $0\leq j\leq g$. Let $\mathbb P^{d_1}_k=\{(x_0,...,x_{d_1},0,...,0)\}$ and $\mathbb P^{d_2}_k=\{(0,...,0,x_{d_1+1},...,x_{d+1})\}$ 
be two disjoint planes in $\mathbb P^{d+1}_k$. As pointed out by Eric Larson, after identifying $Z_i$ with a rational normal curve in $\mathbb P^{d_i}_k$, as in the right part of Figure \ref{fig: remark}, the space $G^r_{\underline d}(X)$ corresponds to the space of $(d-r)$-planes passing though the $g+1$ lines $\overline {Q^1_jQ^2_j}$ modulo the $k^*$ action 
$$y\in k^*\colon (x_0,...,x_{d+1})\mapsto (x_0\cdot y,...,x_{d_1}\cdot y,x_{d_1+1},...,x_{d+1}).$$
Here $\overline {Q^1_jQ^2_j}$ meets $Z_i$ at $Q^i_j$.
Moreover, for a $(d-r)$-plane $\Lambda$ that intersects $\overline{Q^1_jQ^2_j}\backslash \{Q^1_j,Q^2_j\}$ for each $0\leq j\leq g$, if $\Lambda\cap \overline{Q^1_{j_0}Q^2_{j_0}}$ only contains one point, then this point determines the glueing data of a line bundle $L$ of multi-degree $\underline d$ on $X$ at the corresponding node, and each hyperplane in $\mathbb P^{d+1}_k$ that contains $\Lambda$ induces a global section of $L$ up to scaling.

\begin{figure}[H]
\begin{tikzpicture}[scale=0.13]

 \draw[shift={(-70,-3)}][-] (47,-25) -- (55,-48);
    \draw[shift={(-70,-3)}][-] (43,-24) -- (42,-48);
    \draw[shift={(-70,-3)}][-] (40,-20) -- (27,-47);

    \draw[shift={(-70,-3)}][thick] plot [smooth,tension=1] coordinates{(40,-10)(38.5,-24.5)(56,-35)(45,-43)(23,-42)};
    
     \draw [shift={(-70,-3)}](38.2,-23.8) node[circle, fill=black, scale=0.3, label=left:{$N^1_g$}]{};
     
     \draw [shift={(-70,-3)}](42.8,-28.3) node[circle, fill=black, scale=0.3, label=left:{}]{};
     \draw [shift={(-70,-3)}](42.8,-30) node[circle, fill=black, scale=0, label=left:{$N^1_2$}]{};
     
     \draw [shift={(-70,-3)}](48.8,-30.5) node[circle, fill=black, scale=0.3, label=above:{}]{};
      \draw [shift={(-70,-3)}](50,-32.5) node[circle, fill=black, scale=0, label=left:{$N^1_1$}]{};
     
     \draw [shift={(-70,-3)}](52.5,-41) node[circle, fill=black, scale=0.3, label=left:{}]{};
     \draw [shift={(-70,-3)}](54,-44) node[circle, fill=black, scale=0, label=left:{$N^2_1$}]{};
     
     \draw [shift={(-70,-3)}](42.2,-43.5) node[circle, fill=black, scale=0.3, label=above:{}]{};
     \draw [shift={(-70,-3)}](42.2,-46) node[circle, fill=black, scale=0, label=left:{$N^2_2$}]{};
     
     \draw [shift={(-70,-3)}](29,-43) node[circle, fill=black, scale=0.3, label=above:{}]{};
      \draw [shift={(-70,-3)}](27,-43) node[circle, fill=black, scale=0, label=above:{$N^2_g$}]{};
     
     \draw [shift={(-70,-3)}](40,-10) node[circle, fill=black, scale=0, label=above:{$\widetilde C'$}]{};

 \draw [shift={(-70,-3)}](37,-40) node[circle, fill=black, scale=0, label=above:{$\cdots$}]{};

    \draw[-] (10.8,-13.8) -- (10.8,-36.1);
    \draw[-] (10.8,-36.1) -- (23.2,-31.1);
        \draw[-] (23.2,-31.1) -- (23.2,-8.9);
    \draw[-] (10.8,-13.8) -- (23.2,-8.9);
    \draw[-] (24.3,-35.1) -- (19.3,-52.8);
    \draw[-] (19.3,-52.8) -- (58.2,-46.3);
    \draw[-] (62.4,-27.6) -- (58.2,-46.3);
    \draw[-] (62.4,-27.6) -- (24.3,-35.1);
    \draw[-] (16.4,-25.7) -- (36.3,-47.9);
        \draw[-] (16,-24.1) -- (42.1,-41.3);
    \draw[-] (14.6,-20.6) -- (54.1,-36.5);

    \draw[thick] plot [smooth,tension=1] coordinates{(19.4,-15.7)(17.3,-22.8)(19.3,-27.9)(14,-32.6)};
    \draw[thick] plot [smooth,tension=1] coordinates{(58,-35.2)(45,-34)(30,-48)(25,-40)};

      \draw (17.5,-21.7) node[circle, fill=black, scale=0.3, label=above:{}]{};
      \draw (15.5,-21.5) node[circle, fill=black, scale=0, label=above:{$Q^1_g$}]{};

    \draw (19,-26) node[circle, fill=black, scale=0.3, label=above:{}]{};
     \draw (18.5,-25.5) node[circle, fill=black, scale=0, label=right:{$Q^1_1$}]{};

    \draw (19,-28.7) node[circle, fill=black, scale=0.3, label=above:{}]{};
    \draw (18.5,-28.7) node[circle, fill=black, scale=0, label=left:{$Q^1_0$}]{};

     \draw (33.9,-45.2) node[circle, fill=black, scale=0.3, label=above:{}]{};
     \draw (33.9,-45.2) node[circle, fill=black, scale=0, label=right:{$Q^2_0$}]{};

      \draw (38.7,-39) node[circle, fill=black, scale=0.3, label=above:{}]{};
      \draw (38,-39) node[circle, fill=black, scale=0, label=above:{$Q^2_1$}]{};

       \draw (46.5,-33.4) node[circle, fill=black, scale=0.3, label=above:{}]{};
       
       \draw (46.5,-33.4) node[circle, fill=black, scale=0, label=below:{$Q^2_g$}]{};
       
        \draw (33,-30) node[circle, fill=black, scale=0, label=below:{$\cdots$}]{};

        \draw (10,-14) node[circle, fill=black, scale=0, label=above:{$\mathbb P^{d_1}_k$}]{};

         \draw (62,-46) node[circle, fill=black, scale=0, label=above:{$\mathbb P^{d_2}_k$}]{};

         \draw (50,-20) node[circle, fill=black, scale=0, label=above:{$\mathbb P^{d+1}_k$}]{};

          \draw (20,-16) node[circle, fill=black, scale=0, label=above:{$Z_1$}]{};
           \draw (58,-35) node[circle, fill=black, scale=0, label=above:{$Z_2$}]{};
\end{tikzpicture}
\caption{}\label{fig: remark}
\end{figure}

However, unlike in \cite{griffiths1980variety} the intersection $\Sigma=\sigma(\overline{Q^1_0Q^2_0})\cap\cdots \cap\sigma(\overline{Q^1_gQ^2_g})$ in $Gr(d-r+1,d+2)$ is not necessarily dimensionally transverse, where $\sigma(\overline{Q^1_jQ^2_j})$ is the space of $(d-r)$-planes in $\mathbb P^{d+1}_k$ that meets $\overline{Q^1_jQ^2_j}$:   if $g$ is big enough, then the expected dimension of the intersection is negative, but it always contains the $(d-r)$-planes that contain either the given $\mathbb P^{d_1}_k$ or $\mathbb P^{d_2}_k$. Even if we ignore these ``extra" $(d-r)$-planes, the degeneration process did not work well as in \cite{griffiths1980variety}. Using the same symbol as in the introduction (or see the left part of Figure \ref{fig: remark}), the main fact in \cite{griffiths1980variety} is that for general points $N^i_{m+1}$, if a $(d-r-1)$-plane $\Lambda'$ in $\mathbb P^{d}_k$ satisfies $\dim (\Lambda'\cap \overline{aN^1_1})\geq \lambda$ for some $a\leq d-1$ and $\Lambda'\cap \overline {N^1_{m+1}N^2_{m+1}}\neq \emptyset$, then  $\dim(\Lambda'\cap \overline {aN^1_1,N^1_{m+1},N^2_{m+1}})\geq \lambda +1$. Here $\overline {aN^1_1,N^1_{m+1},N^2_{m+1}}$ denotes the linear subspace of $\mathbb P^d_k$ spanned by $\overline {aN^1_1}$ and $N^i_{m+1}$, and recall that $\overline {aN^1_1}$ is the $a$-th osculating plane of $\widetilde C'$ at $N^1_1$. 
 In our case, after degenerating $Q^i_1,...,Q^i_m$ to $Q^i_0$ for $i=1,2$, from the first $m+1$ items of the intersection of $\Sigma$ we get Schubert varieties of flags of the form
 $$\overline {b^1_0Q^1_0,b^2_0Q^2_0}\subsetneq\cdots\subsetneq\overline {b^1_{d+1}Q^1_0,b^2_{d+1}Q^2_0}.$$
Note that $b^1_\bullet$ and $b^2_\bullet$ are non-decreasing and non-negative integers bounded above by $d_1+1$ and $d_2+1$ respectively, and again, $\overline {b^1_jQ^1_0,b^2_jQ^2_0}$ denotes the space spanned by $\overline {b^1_jQ^1_0}$ and $\overline {b^2_jQ^2_0}$. It is not true in general that $\dim(\Lambda\cap \overline{b^1_lQ^1_0,b^2_lQ^2_0})\geq \lambda$ for some $l\leq d$ and $\Lambda\cap \overline {Q^1_{m+1}Q^2_{m+1}}\neq \emptyset$ implies  
 \begin{equation}\label{eq:introduction}
    \dim( \Lambda\cap \overline {b^1_lQ^1_0,b^2_lQ^2_0,Q^1_{m+1},Q^2_{m+1}})\geq \lambda +1,
 \end{equation}
 due to the possible existence of planes $\Lambda$ containing $Q^i_{m+1}$ when $b^i_l=d_i+1$ for either $i=1$ or $i=2$. Moreover, the locus of $\Lambda$ not satisfying (1) may have dimension more than expected. 
For example, consider the Schubert cycle $S$ consisting of $(d-r)$-planes that meet $\mathbb P^{d_1}_k$ in dimension at least $x $. The intersection of $S$ with $\sigma(\overline{Q^1_{m+1}Q^2_{m+1}})$ is expected to have codimension $\gamma=r+(x+1)(r+1+x-d_1)$ in $Gr(d-r+1,d+2)$. But the subset of $S$ of planes containing $Q^1_{m+1}$ has codimension $r+1+x(r+1+x-d_1)$, which is less than $\gamma$ if $r+x>d_1$.

It turns out that the space $\Sigma/k^*$, as a potential compactification of $G^r_{\underline d}(X)$, is too big for the degeneration argument, and we need to rule out the ``bad planes": $(d-r)$-planes in $\Sigma$ which contain exactly one of $\{Q^1_j,Q^2_j\}$ for some $j$. Note that intuitively, ``bad planes" correspond to torsion free sheaves on $X$ that are not locally free at one node or more. We thus adopt Esteve's compactification of the Jacobian of $X$, which is obtained by adding certain controllable torsion free sheaves.

Now the Schubert conditions above for a $(d-r)$-plane $\Lambda\subset \mathbb P^{d+1}_k$ induce ramification conditions for the corresponding linear series. More precisely, let $\widetilde X$ be the normalization of $X$ and $(L,V)$ be a linear series on $\widetilde X$ induced by a $\Lambda$. Then $\dim(\Lambda\cap \overline{b^1_lQ^1_0,b^2_lQ^2_0})\geq \lambda$ if and only if $$\dim(V(-b^1_lQ^1_0-b^2_lQ^2_0))\geq r+1-(b^1_l+b^2_l-1)+\lambda,$$ 
and $\Lambda\cap \overline {Q^1_{m+1}Q^2_{m+1}}\neq \emptyset$ if and only if $\dim V(-Q^1_{m+1}-Q^2_{m+1})\geq r$. These two inequalities imply immediately that 
$$\dim V(-b^1_lQ^1_0-b^2_lQ^2_0-Q^1_{m+1}-Q^2_{m+1})\geq r-(b^1_l+b^2_l-1)+\lambda,$$
which is equivalent to $(1)$ if $b^i_l\leq d_i$ for each $i$. Therefore, we are able to modify the proof in \cite{griffiths1980variety} and apply to our case by considering linear series with imposed ramification.
\end{rem}

 \bibliographystyle{amsalpha}\bibliography{1}

\end{document}